\documentclass[11pt]{amsart}

\usepackage{amssymb}
\usepackage{amsmath}
\usepackage[active]{srcltx}
\usepackage{t1enc}
\usepackage[latin2]{inputenc}
\usepackage{verbatim}
\usepackage{amsmath,amsfonts,amssymb,amsthm}
\usepackage[mathcal]{eucal}
\usepackage{enumerate}
\usepackage[centertags]{amsmath}
\usepackage{graphics}

\newtheorem{theorem}{Theorem}
\newtheorem{conjecture}{Conjecture}
\newtheorem{lemma}{Lemma}

\begin{document}
\author{Davit Baramidze}
\title[partial sums ]{On some weighted maximal operators of partial sums of Walsh-Fourier series in the space $H_1$}

\address{D. Baramidze, The University of Georgia, School of science and technology, 77a Merab Kostava St, Tbilisi 0128, Georgia and Department of Computer Science and Computational Engineering, UiT - The Arctic University of Norway, P.O. Box 385, N-8505, Narvik, Norway.}
\email{ datobaramidze20@gmail.com \ \ \ davit.baramidze@ug.edu.ge   }

\thanks{The research was supported by Shota Rustaveli National Science Foundation grant no. PHDF-21-1702.}
\maketitle

\begin{abstract}
In the first part of this paper we describe the status of the art of this subject. In the second part we present and motivate some new results. Indeed,
we  introduce some new weighted  maximal operators of the partial sums of the Walsh-Fourier series. We prove that for some "optimal" weights these new operators indeed are bounded  from the martingale Hardy space $H_{1}(G)$ to the space $\text{weak}-L_{1}(G),$ but is not bounded from $H_{1}(G)$ to the space $L_{1}(G).$

\end{abstract}

\date{}



\section{Introduction and status of the art}

It is well-known that  the Walsh system does
not form a basis in the space $L_1$ (see e.g. \cite{Bary}). Moreover, there exists a function in the dyadic Hardy space $H_{1}(G),$ such that the partial sums of $f$ are not bounded in the  $L_{1}$-norm.
Uniform and pointwise convergence and some approximation properties of
partial sums in $L_{1}(G)$ norms were investigated by  Avdispahi\'c and Memi\'c \cite{am},  G\'at, Goginava and Tkebuchava \cite{ggt,gt}, Nagy \cite{na1,NT1}, Onneweer \cite{Onn10} , Paley \cite{Pal11}, and Persson, Schipp,  Tephnadze and Weisz \cite{PSTW}. Fine \cite{fi1} obtained sufficient conditions for the uniform convergence
which are completely analogous to the Dini-Lipschits conditions. Gulicev \cite{Gul1} estimated the rate of uniform convergence of a Walsh-Fourier series by using Lebesgue constants and modulus of continuity. These problems for
Vilenkin groups were investigated by Blatota, Nagy, Persson and Tephnadze \cite{BNPT} (see also \cite{Blahota1,BPST2,BNT9,BNT10,PTW2}), Fridli \cite{Fr1}, G\'at  \cite{Gat1}, Simon \cite{Simon1} and Tephnadze \cite{MST,tep5,tep8}. 

In the study of convergence of subsequences of  partial sums and their restricted  maximal operators on the martingale Hardy spaces $H_p$ for $0<p\leq 1,$ the central role is played by the fact that any natural number $n\in \mathbb{N}$ can be uniquely expressed as
$
n=\sum_{k=0}^{\infty }n_{j}2^{j},  \ \ n_{j}\in Z_{2} ~(j\in \mathbb{N}), 
$
where only a finite numbers of $n_{j}$ differ from zero
and their important characters  $\left[ n\right],$ $\left\vert n\right\vert,$ $\rho\left( n\right)$  and $V(n)$ are defined by
\begin{equation*}
\left[ n\right] :=\min \{j\in \mathbb{N},n_{j}\neq 0\}, \ \ \  \ \
\left\vert n\right\vert :=\max \{j\in \mathbb{N},n_{j}\neq 0\}, \ \ \ \ \
\rho\left( n\right) =\left\vert n\right\vert -\left[ n\right]
\end{equation*}
and
\begin{equation*}
V\left( n\right): =n_{0}+\overset{\infty }{\underset{k=1}{\sum }}\left|
n_{k}-n_{k-1}\right|, \text{ \ for \
	all \ \ }n\in \mathbb{N}.
\end{equation*}
Moreover, every $n\in \mathbb{N}$ can be also represented as 
$n=\sum_{i=1}^{r}2^{n_{i}},n_{1}>n_{2}>...n_{r}\geq 0$ and for any $\{n_{s_j}\}, j=1,2,\ldots,r,$ satisfying
$
2^{s}\le {{n}_{{{s}_{1}}}}\le {{n}_{{{s}_{2}}}}\le ...\le {{n}_{{{s}_{r}}}}< {2^{s+1}}, \ s\in \mathbb{N},  
$
we  define numbers
\begin{equation}\label{rho}
s_{-} :=\min \{\left[n_{s_j}\right]\}, \ \ \
s_{+} :=\max \{\left\vert n_{s_j}\right\vert\}=s, \ \ \ 
\rho_s\left( n_{s_j}\right) :=s_{+}-s_{-}.
\end{equation}

In particular, (see \cite{BPT1}, \cite{luk} and  \cite{sws}) 
\begin{equation*}
V\left( n\right)/8\leq \Vert D_n\Vert_1\leq V\left( n\right)
\end{equation*}
Hence, for any $F\in L_1,$ there exists an absolute constant $c$ such that 
\begin{equation*}
\left\Vert S_n F\right\Vert_1\leq c{V\left( n\right)
}\left\Vert F\right\Vert_1.
\end{equation*}
Moreover, for any $f\in H_1$  (see \cite{tep2})
\begin{equation*}
\left\Vert S_{n}F\right\Vert _{H_{1}}\leq c{V\left( n\right)}\left\Vert F\right\Vert _{H_{1}}.
\end{equation*}

For $0<p<1$ in  \cite{tep0,tep1}  the weighted  maximal operator $\overset{\sim }{S }^{*,p},$ defined by
\begin{equation}\label{max0}
\overset{\sim }{S }^{*,p}F:=\sup_{n\in\mathbb{N}}\frac{\left|S _{n}F\right|} {\left( n+1\right)^{1/p-1} }
\end{equation}
was investigated and it was proved that the following inequalities hold:
\begin{equation*}
\left\Vert \overset{\sim }{S }^{*}F\right\Vert_{p}\leq c_{p}\left\Vert F\right\Vert _{H_{p}}
\end{equation*}
Moreover, it was also proved that the rate of the sequence $\{\left(n+1\right)^{ 1/p-1}\}$ given in the denominator of \eqref{max0} can not be improved.

In \cite{tep2} and \cite{tep3}  it was proved that if $F\in H_{p},$ then there exists an absolute constant  $c_{p},$ depending only on $p,$ such that
\begin{equation*}
	\text{ }\left\Vert S_{n}F\right\Vert _{H_{p}}\leq c_{p}2^{\rho\left( n\right)\left( 1/p-1\right) }\left\Vert F\right\Vert _{H_{p}}.
\end{equation*}	
In \cite{BPST} it was proved that the maximal operator  
\begin{equation}\label{1010a} 
\sup_{n\in \mathbb{N}}\frac{\left\vert S_{n}F\right\vert}{2^{\rho \left(n\right) \left( 1/p-1\right)}}
\end{equation}
is bounded from $H_p$ to $\text{weak}-L_p.$ 
Moreover, if $0<p<1,$ $\left\{ n_{k}:\text{ }k\geq 0\right\} $ is any increasing sequence of positive integers such that 
$
\sup_{k\in \mathbb{N}}\rho\left( n_{k}\right) =\infty  
$
and  
$\Phi :\mathbb{N}_{+}\rightarrow \mathbb{R}$ 
is any nondecreasing function, satisfying the condition 
\begin{equation*}	
\overline{\underset{k\rightarrow \infty }{\lim }}\frac{2^{\rho\left(
n_{k}\right) \left( 1/p-1\right) }}{\Phi \left( n_{k}\right) }=\infty,
\end{equation*}
then there exists a martingale $F\in H_{p}(G),$ such that
\begin{equation*}
\underset{k\in \mathbb{N}}{\sup }\left\Vert \frac{S_{n_{k}}F}{\Phi \left(n_{k}\right) }\right\Vert_{\text{weak}-L_p}=\infty .
	\end{equation*}

In \cite{BPST} it was also proved that the weighted maximal operator \eqref{1010a}  is not bounded from $H_p(G)$ to the Lebesgue space $L_p(G),$ for $0<p<1.$

In \cite{Bara3} it was proved that if  $0<p<1,$ $f\in {{H}_{p}},$ $\left\{ n_{k}:k\geq 0\right\} $ is a sequence of positive integers and 
	$\left\{ n_{s_i}:\ 1\leq i\leq r\right\} \subset \left\{ n_{k}:\ k\geq 0\right\} $ are integers such that
	$
	2^{s}\le {{n}_{{{s}_{1}}}}\le {{n}_{{{s}_{2}}}}\le ...\le {{n}_{{{s}_{r}}}}\le {2^{s+1}}, \  s\in \mathbb{N}, $
	then the weighted maximal operator $\widetilde{S }^{\ast ,\nabla },$ defined by
	\begin{equation*}
	\widetilde{S }^{\ast ,\nabla }F:=\sup_{s\in\mathbb{N}}\sup_{2^s\leq n_{s_i}< 2^{s+1}}
	\frac{\left\vert S _{n}F\right\vert}{2^{\rho_s\left( n_{s_i}\right) \left( 1/p-1\right)}},
	\end{equation*}
	where  $\rho_s\left( n_{s_i}\right)$ are defined by \eqref{rho},
	is bounded from the Hardy space ${{H}_{p}}$ to the Lebesgue space ${{L}_{p}}$. Moreover, if $0<p<1,$ $\left\{ n_{k}:k\geq 0\right\} $ is a sequence of positive numbers and 
	$\left\{ n_{s_i}:\ 1\leq i\leq r\right\} \subset \left\{ n_{k}:\ k\geq 0\right\} $ is a subsequence satisfying the condition
	$
	2^{s}\le {{n}_{{{s}_{1}}}}\le {{n}_{{{s}_{2}}}}\le ...\le {{n}_{{{s}_{r}}}}\le {2^{s+1}}, \  s\in \mathbb{N}, $
	then, for any nonnegative, nondecreasing function
	$\varphi:\mathbb{N}_+\to \mathbb{R}$  satisfying condition 
	\begin{equation*}
	\sup_{s\in\mathbb{N}}\sup_{2^s\leq n_{s_i}< 2^{s+1}}\frac{2^{\rho_s\left( n_{s_i}\right) \left( 1/p-1\right)}}{\varphi(n_{s_i})}=\infty,
	\end{equation*}
	the maximal operator, defined by
	\begin{equation*}
	\sup_{s\in\mathbb{N}}\sup_{2^s\leq n_{s_i}< 2^{s+1}}
	\frac{\left\vert S _{n}F\right\vert}{\varphi\left( n_{s_i}\right)},
	\end{equation*}
is not bounded from the Hardy space ${{H}_{p}}$ to the Lebesgue space ${{L}_{p}}$.

In \cite{BPT} it was proved that if $0<p<1$, $f\in {{H}_{p}}\left(G \right)$ and $\varphi:\mathbb{N}_+\to \mathbb{R}_+$ be any nonnegative and nondecreasing function satisfying the condition
	\begin{equation*}
	\sum_{n=1}^{\infty}\frac{1}{\varphi^p(n)}<c<\infty,
	\end{equation*}
	then, for any sequence  $\left\{ n_{k}:k\geq 0\right\} $  of positive integers, the weighted maximal operator $\widetilde{S }^{\ast ,\nabla },$ defined by
	\begin{equation}\label{maxoperator1}
	\widetilde{S }^{\ast ,\nabla }F=\underset{k\in \mathbb{N}}{\sup }
	\frac{\left\vert S_{n_k}F\right\vert}{2^{\rho\left( n_k\right) \left( 1/p-1\right)}\varphi(\rho\left( n_k\right))},
	\end{equation}
	is bounded from the Hardy space ${{H}_{p}}$ to the Lebesgue space ${{L}_{p}}$.
Moreover, for any $0<p<1$ and any sequence $\left\{ n_{k}:k\geq 0\right\} $ of positive numbers and
	$\varphi:\mathbb{N}_+\to \mathbb{R}_+$ be any nonnegative and nondecreasing function satisfying the condition 
	\begin{equation*}
	\sum_{n=1}^{\infty}\frac{1}{\varphi^p(n)}=\infty,
	\end{equation*}
	the weighted maximal operator $\widetilde{S }^{\ast ,\nabla },$ defined by \eqref{maxoperator1}, is not bounded from the Hardy space ${{H}_{p}}(G)$ to the Lebesgue space ${{L}_{p}}(G)$.
Hence, we get that if  $0<p<1$ and $f\in {{H}_{p}}\left(G \right)$, then the weighted maximal operator $\widetilde{S }^{\ast ,\nabla, \varepsilon },$ defined by
\begin{equation*}
\widetilde{S }^{\ast ,\nabla, \varepsilon }F:=\underset{n\in \mathbb{N}}{\sup }\frac{\left\vert S_{n}F\right\vert}{2^{\rho\left( n\right) \left( 1/p-1\right)}\left(\rho\left( n\right)\log^{1+\varepsilon}\rho\left( n\right)\right)^{1/p}}, \ \ \ \varepsilon>0,
\end{equation*}
is bounded from the Hardy space ${{H}_{p}}(G)$ to the Lebesgue space ${{L}_{p}}(G)$ for $\varepsilon>0$ and is not bounded from the Hardy space ${{H}_{p}}(G)$ to the Lebesgue space ${{L}_{p}}(G)$ for $\varepsilon=0.$

In this paper we prove that some new weighted  maximal operators of the partial sums of the Walsh-Fourier series with "optimal" weights  are bounded  from the martingale Hardy space $H_{1}(G)$ to the space $\text{weak}-L_{1}(G),$ but is not bounded from $H_{1}(G)$ to the space $L_{1}(G).$

\section{Preliminaries}

\bigskip Let $\mathbb{N}_{+}$ denote the set of the positive integers, $%
\mathbb{N}:=\mathbb{N}_{+}\cup \{0\}.$ Denote by $Z_{2}$ the discrete cyclic
group of order $2$, that is $Z_{2}:=\{0,1\},$ where the group operation is the modulo $2$ addition and every subset is open. The Haar measure on $Z_{2}$ is given so that the measure of a singleton is $1/2$.
Define the group $G$ as the complete direct product of the group $Z_{2},$
with the product of the discrete topologies of $Z_{2}$`s. The elements of $G$
are represented by sequences 
$x:=(x_{0},x_{1},...,x_{j},...), \ \text{ where } \ 
x_{k}=0\vee 1.$

It is easy to give a base for the neighborhood of $x\in G:$ 
\begin{equation*}
I_{0}\left( x\right) :=G,\text{ \ }I_{n}(x):=\{y\in
G:y_{0}=x_{0},...,y_{n-1}=x_{n-1}\}\text{ }(n\in \mathbb{N}).
\end{equation*}

Denote 
$I_{n}:=I_{n}\left( 0\right) , \ \  \overline{I_{n}}:=G\backslash 
I_{n}.$  
Then, it is easy to prove that
\begin{equation}\label{2}
\overline{I_{M}}=\overset{M-1}{\underset{s=0}{\bigcup }}I_{s}\backslash
I_{s+1}.  
\end{equation}

Let define the Walsh system $w:=(w_{n}:n\in \mathbb{N})$ on $G$ by
\begin{equation*}
w_{n}(x):=\overset{\infty }{\underset{k=0}{\Pi }}r_{k}^{n_{k}}\left(
x\right) = \left( -1\right) ^{%
\underset{k=0}{\overset{\left\vert n\right\vert }{\sum }}n_{k}x_{k}}\text{%
\qquad }\left( n\in \mathbb{N}\right) .
\end{equation*}

The Walsh system is orthonormal and complete in $L_{2}(G) $ (see
e.g.  \cite{tepbook} and \cite{sws}).

If $f\in L_{1} $ we define Fourier coefficients, partial sums of the Fourier series, Dirichlet kernels with respect
to the Walsh system by
\begin{eqnarray*}
\widehat{f}\left(k\right):=\int_{G}fw_{k}d\mu \ \left( k\in \mathbb{N}\right), \ \ 
S_{n}f:=\sum_{k=0}^{n-1}\widehat{f}\left( k\right) w_{k}, \ \   
D_{n}:=\sum_{k=0}^{n-1}w_{k} \ \left( n\in \mathbb{N}_{+}\right). \notag
\end{eqnarray*}

Recall that (see \cite{gol}, \cite{tepbook} and \cite{sws})
\begin{equation}\label{1dn}
D_{2^{n}}\left( x\right) =\left\{ 
\begin{array}{ll}
2^{n}, & \,\text{if\thinspace \thinspace \thinspace }x\in I_{n}, \\ 
0, & \text{if}\,\, \ x\notin I_{n},
\end{array}
\right.  
\end{equation}
\begin{equation}\label{2dn}
D_{n}=w_{n}\overset{\infty }{\underset{k=0}{\sum }}n_{k}r_{k}D_{2^{k}}=w_{n}%
\overset{\infty }{\underset{k=0}{\sum }}n_{k}\left(
D_{2^{k+1}}-D_{2^{k}}\right) ,\text{ for \ }n=\overset{\infty }{\underset{i=0%
	}{\sum }}n_{i}2^{i}.  
\end{equation}

The $\sigma $-algebra generated by the intervals $\left\{ I_{n}\left(
x\right) :x\in G\right\} $ will be denoted by $\zeta _{n}\left( n\in \mathbb{N}\right).$ 
Denote by $F=\left( F_{n},n\in \mathbb{N}\right) $
martingale with respect to $\digamma _{n}$ $\left( n\in \mathbb{N}\right) $
(see e.g. \cite{We1,We3}).
The maximal function $F^{\ast }$ of a martingale $F$ is defined by
$
F^{\ast }:=\sup_{n\in \mathbb{N}}\left\vert F_{n}\right\vert .
$
For $0<p<\infty $ the Hardy martingale spaces $H_{p}\left( G\right) $
consists of all martingales for which
\begin{equation*}
\left\Vert F\right\Vert _{H_{p}}:=\left\Vert F^{\ast }\right\Vert
_{p}<\infty .
\end{equation*}

For every martingale $F=\left( F_{n},n\in \mathbb{N}
\right) $ and every $k\in \mathbb{N}$ the limit
\begin{equation*}
\widehat{F}\left( k\right) :=\lim_{n\rightarrow \infty }\int_{G}F_{n}\left(
x\right) w_{k}\left( x\right) d\mu \left( x\right)  
\end{equation*}
exists and it is called the $k$-th Walsh-Fourier coefficients of $F.$

A bounded measurable function $a$ is called $p$-atom, if there exists a dyadic
interval $I,$ such that 
\begin{equation*}
\int_{I}ad\mu =0,\text{ \ \ }\left\Vert a\right\Vert _{\infty }\leq \mu
\left( I\right) ^{-1/p},\text{ \ \ supp}\left( a\right) \subset I.
\end{equation*}

The dyadic Hardy martingale spaces $H_{p}(G)$ for $0<p\leq 1$ have an atomic
characterization. Namely, the following holds (see \cite{tepbook}, \cite{We1,We3}):

\begin{lemma}\label{W1}
	A martingale $F=\left( F_{n},n\in \mathbb{N}\right) $ belongs to $H_{p}(G) \ \left( 0<p\leq 1\right) $ if and only if there exists a sequence $%
	\left( a_{k},\text{ }k\in \mathbb{N}\right) $ of p-atoms and a sequence $%
	\left( \mu _{k},k\in \mathbb{N}\right) $ of  real numbers such that for
	every $n\in \mathbb{N},$
	\begin{equation}
	\qquad \sum_{k=0}^{\infty }\mu _{k}S_{2^{n}}a_{k}=F_{n},\text{ \ \ \ }%
	\sum_{k=0}^{\infty }\left\vert \mu _{k}\right\vert ^{p}<\infty ,  \label{2A}
	\end{equation}
	Moreover, 	
	$
	\left\Vert F\right\Vert _{H_{p}}\backsim \inf \left( \sum_{k=0}^{\infty
	}\left\vert \mu _{k}\right\vert ^{p}\right) ^{1/p},
	$
	where the infimum is taken over all decomposition of $F$ of the form (\ref%
	{2A}).
\end{lemma}

\section{The main result}

Our  main result reads:
\begin{theorem}\label{theorem111}
	a) 	Let $f\in {{H}_{1}}$. Then the weighted maximal operator $\widetilde{S }^{\ast ,\nabla },$ defined by 
\begin{equation}\label{snmax}
\widetilde{S }^{\ast ,\nabla }:=\underset{k\in \mathbb{N}}{\sup }
\frac{\left\vert S_{n_k}F\right\vert}{V( n_k)}
\end{equation}	
	 is bounded from the Hardy space ${{H}_{1}}$ to the space $\text{weak}-L_1.$
	 
   b)	Let $f\in {{H}_{1}}$. Then the weighted maximal operator $\widetilde{S }^{\ast ,\nabla },$ defined by \eqref{snmax}	
   is not bounded from the Hardy space ${{H}_{1}}$ to the space $L_1.$
\end{theorem}
\begin{proof}
	Since $S _{n}/V(n)$ is bounded from $L_{\infty }$ to $%
	L_{\infty },$
	by Lemma \ref{W1}, the proof of Theorem \ref{theorem111}
	will be complete, if we prove that
	\begin{equation}\label{weaktypesigma5}
	t\mu\left\{x\in\overline{I_M}:\widetilde{S }^{\ast ,\nabla }a(x) \geq t\right\} \leq c_p<\infty,
	\text{ \ \ \ \ }t\geq 0
	\end{equation}
	for every $p$-atom $a.$ In this parer $c_p$ (or $C_p$) denotes a positive constant depending only on $p$ but which can be different in different places.

	We may assume that $a$ is an arbitrary $p$-atom,
	with support $I, \ \mu \left( I\right) =2^{-M}$ and $I=I_{M}.$ 
	It is easy to see that 
	$S _{n}a\left( x\right) =0, \  \text{ when } \  n< 2^{M}.$
	Therefore, we can suppose that $n\geq 2^{M}.$
	Since $\left\Vert a\right\Vert _{\infty }\leq 2^{M},$
	we obtain that 
	\begin{eqnarray*}
		\frac{\left\vert S _{n}a\left( x\right) \right\vert}{V(n)} 
		&\leq& \frac{1}{V(n)} 
		\left\Vert a\right\Vert _{\infty }\int_{I_{M}}\left\vert D_{n}\left(
		x+t\right) \right\vert \mu \left( t\right)
		\leq 2^{M}\int_{I_{M}}\left\vert
		D_{n}\left( x+t\right) \right\vert \mu \left( t\right) .
	\end{eqnarray*}

	Let $I_{s}\backslash I_{s+1} .$ 
	Then, it is easy to see that $x+t\in I_{s}\backslash I_{s+1}$ for $t\in I_M$ and if we again combine  \eqref{1dn} and \eqref{2dn} we find that 
	$D_{n}\left( x+t\right)\leq c2^s, \  \text{ for } \  t\in I_{M}$
	and 
	\begin{eqnarray}\label{12}
	\frac{\left\vert S_{n}a\left( x\right) \right\vert }{V(n)} 
	\leq c_p 2^{M}2^{s-M} \leq c_{p}2^{s}.
	\end{eqnarray}
	
	By applying (\ref{12})  for any  $x\in I_{s}\backslash I_{s+1}  ,\,0\leq s< M,$ we find that
	\begin{eqnarray}\label{weaktypesigma1}
	\widetilde{S }^{\ast,\nabla }a\left( x\right)=\sup_{n\in\mathbb{N}}\left(\frac{\left\vert S_{n}a\left( x\right) \right\vert}{2^{\rho\left( n\right) \left( 1/p-1\right)}} \right) \leq C_{p}2^{s}.
	\end{eqnarray}
	
	It immediately follows that for  $s\leq M$ we have the following estimate
	\begin{equation*}
	\widetilde{S}^{\ast ,\nabla }a\left( x\right) \leq C_p2^{M}\text{ \ \ for any\ \ }x\in I_{s}\backslash I_{s+1}, \ \ s=0,1,\cdots, M
	\end{equation*}%
	and also that
	\begin{equation}\label{weaktypesigma0}
	\mu \left\{ x\in I_{s}\backslash I_{s+1}:\widetilde{S }^{\ast ,\nabla }a\left( x\right)> C_p2^{k}\right\} =0, \ \ \ k=M, M+1,\ldots
	\end{equation}
	
	By combining \eqref{2} and \eqref{weaktypesigma1} we get that
	\begin{eqnarray*}
		\left\{ x\in \overline{I_{N}}:\widetilde{S }^{\ast ,\nabla }a\left( x\right)\geq C_p2^{k}\right\} 
		\subset \bigcup_{s=k}^{M-1}\left\{ x\in I_{s}\backslash
		I_{s+1}:\widetilde{S }^{\ast ,\nabla }a\left( x\right)\geq C_p 2^{k}\right\} 
	\end{eqnarray*}
	and
	\begin{eqnarray}\label{weaktypesigma4}
	\mu \left\{ x\in \overline{I_{M}}:\widetilde{S }^{\ast ,\nabla }a\left( x\right)\geq C_p 2^{k}\right\}\leq  \overset{M-1}{\underset{s=k}{\sum }}\frac{1}{2^{s}}\leq \frac{2}{2^{k}}.
	\end{eqnarray}
	In view of \eqref{weaktypesigma0} and \eqref{weaktypesigma4} we can conclude that
	\begin{eqnarray*}
		2^{k}\mu \left\{ x\in \overline{I_{N}}:\widetilde{S }^{\ast ,\nabla }a\left( x\right)\geq {C_p}{2^{k}}\right\}<c_p<\infty,
	\end{eqnarray*}
	which shows that \eqref{weaktypesigma5} holds and the proof of part a) is complete.
	
Set
\begin{equation*}
f_{n_{k}}\left( x\right) =D_{2^{n_{k}+1}}\left( x\right)
-D_{2^{{n_{k}}}}\left( x\right) ,\text{ \qquad }n_{k}\geq 3.
\end{equation*}

In \cite{BPST} it was proved that
$
\left\Vert f_{n_{k}}\right\Vert _{H_{p}} \leq 1
$
and
$
\left\vert S_{2^{n_k}+2^s}f_{n_k}\left( x\right) \right\vert
\geq c2^{s}, $ for $ x\in I_{s+1}\left(e_{s}\right), $ $ s=0,\cdots , n_k-1.
$
Hence,
\begin{eqnarray*}\label{22abc}
\int_{G}\sup_{n\in \mathbb{N}}\frac{\left\vert S_{n}f_{n_{k}} \right\vert }{V\left(n\right)}d\mu 
\geq \overset{n_{k}-1}{\underset{s=0}{\sum }}
\int_{I_{s+1}\left(e_{s}\right)} \frac{\left\vert S _{2^{n_k}+2^s}f_{n_{k}} \right\vert }{V\left(2^{n_k}+2^s\right)} d\mu
 \geq c_p\overset{n_{k}-1}{\underset{s=0}{\sum }}\frac{1}{2^{s}}2^{s}\geq C_p n_k .
\end{eqnarray*}
Finally, we get that
\begin{eqnarray*}
	\frac{\left( \int_{G}\left(\sup_{n\in \mathbb{N}}\frac{\left\vert S_{n}f_{n_{k}}\left( x\right) \right\vert }{V\left(n\right)}\right)d\mu \left( x\right) \right) ^{1/p}}{\left\| f_{n_{k}}\right\| _{H_p}} 
	\geq c_p n^{1/p}_k\rightarrow \infty,\quad \text{as \quad }k\rightarrow \infty,
\end{eqnarray*}

so also part b) is proved and the proof is complete.	
	
\end{proof}

\section{Final comments and open problem}
We finalize this paper with another natural conjecture related to Theorem \ref{theorem111}. For this we need some new characterization of $n\in\mathbb{N}.$

Let 
$
2^{s}\le {{n}_{{{s}_{1}}}}\le {{n}_{{{s}_{2}}}}\le ...\le {{n}_{{{s}_{r}}}}\le {2^{s+1}}, \ s\in \mathbb{N}. $
For such $n_{s_j},$ which can be written as
${{n}_{{{s}_{j}}}}=\sum\limits_{i=1}^{{{r}_{{{s}_{j}}}}}{\sum\limits_{k=l_{i}^{{{s}_{j}}}}^{t_{i}^{{{s}_{j}}}}{2^k}},$
where   
$0\le l_{1}^{{{s}_{j}}}\le t_{1}^{{{s}_{j}}}\le l_{2}^{{{s}_{j}}}-2<l_{2}^{{{s}_{j}}}\le t_{2}^{{{s}_{j}}}\le ...\le l_{{{r}_{j}}}^{{{s}_{j}}}-2<l_{{{r}_{s_j}}}^{{{s}_{j}}}\le t_{{{r}_{s_j}}}^{{{s}_{j}}},$
we define
\begin{eqnarray}\label{As}
{{A}_{s}}&:=&{\left\{ l_{1}^{s},l_{2}^{s},...,l_{r^1_{s}}^{s} \right\}}
\bigcup{\left\{ t_{1}^{s},t_{2}^{s},...,t_{{{r}_{s}^2}}^{s} \right\}}
={\left\{ u_{1}^{s},u_{2}^{s},...,u_{r^3_{s}}^{s} \right\}},
\end{eqnarray}
where $ u_{1}^{s}<u_{2}^{s}<...<u_{r^3_{s}}^{s}.$
We note that $
t_{{{r}_{s_j}}}^{{{s}_{j}}}=s\in {{A}_{s}}, \ \ \text{ for } \ \  j=1,2,...,r.  
$

We denote the cardinality of the set $A_s$ by $\vert A_s\vert$, that is
$card(A_s):=\vert A_s\vert.$
By this definition we can conclude that $\vert A_s\vert=r_s^3\leq r_s^1+r_s^2. $
It is evident that 
$\sup_{s\in\mathbb{N}}\vert A_s\vert<\infty$
if and only if the sets  $\{{{n}_{{{s}_{1}}}},{{n}_{{{s}_{2}}}}, ...,{{n}_{{{s}_{r}}}}\}$ are uniformly finite for all $s\in \mathbb{N}_+$ and each ${{n}_{{{s}_{j}}}}$ has bounded variation
$
V(n_{s_j})<c<\infty, \  \text{for each}  \  j=1,2,\ldots,r.
$

\begin{conjecture}\label{theorem11}
	a) Let $f\in {{H}_{1}}$ and $\left\{ n_{k}:k\geq 0\right\} $ is a sequence of positive numbers. Then the weighted maximal operator $\widetilde{S }^{\ast ,\nabla },$ defined by
	\begin{equation*}
	S^{\ast ,\nabla }F:=\underset{k\in \mathbb{N}}{\sup }
	\frac{\left\vert S _{n_k}F\right\vert}{A_{\vert n_ k\vert }},
	\end{equation*}
	is bounded from the Hardy space ${{H}_{1}}$ to the Lebesgue space ${{L}_{1}}$.

	b) (Sharpness) Let  
	$
	{{\sup }_{k\in \mathbb{N}}}\vert{{A}_{n_k}}\vert=\infty
	$
	and 
	$\{\varphi_n\}$ is a nondecreasing sequence
	satisfying the condition
	$
	\overline{\lim}_{k\rightarrow \infty }\left({A_{\vert n_k\vert} } /{
		\varphi_{\vert n_k\vert}}\right)=\infty .
	$
	Then there exists a martingale $f\in H_1,$ such that the maximal operator, defined by
	$\underset{k\in \mathbb{N}}{\sup }\frac{\left\vert S _{n_k}F\right\vert}{\varphi_{\vert n_k\vert}}$
	is not bounded from the Hardy space ${{H}_{1}}$ to the Lebesgue space ${{L}_{1}}.$
\end{conjecture}


\begin{thebibliography}{99}
\bibitem{am} \textit{M. Avdispahi\'c and N. Memi\'c,} On the Lebesgue test for
convergence of Fourier series on unbounded Vilenkin groups, Acta Math.
Hungar., 129, 4 (2010), 381-392.

\bibitem{Bary} \textit{N. K. Bary,} Trigonometric series. Gos. Izd. Fiz. Mat.
Lit. Moscow 1961, (Russian).

\bibitem{Bara3} \textit{D. Baramidze,} Martingale Hardy spaces and some new weighted maximal operators of partial sums of Walsh-Fourier series, Georgian Math. J., (to appear).

\bibitem{BPT} \textit{D. Baramidze, L.-E. Persson and G. Tephnadze,} Some new $(H_p-L_p)$ type inequalities for weighted maximal operators of partial sums of Walsh-Fourier series, Positivity, 2023, 27(3), 38.

\bibitem{BPST} \textit{	D. Baramidze, L.-E. Persson, H. Singh and G. Tephnadze,} Some new weak $(H_p-L_p)$ type inequality for weighted maximal operators of partial sums of Walsh-Fourier series, Mediterr. J. Math., 20, 5, (2023) 284.

\bibitem{BPST2} \textit{ D. Baramidze, L.-E. Persson, H. Singh and G. Tephnadze,} Some new results and inequalities for subsequences of N\"orlund logarithmic means of Walsh-Fourier series, J. Inequal. Appl.,  (2022), paper no. 30, 13 pp.

\bibitem{Blahota1} \textit{I. Blahota,} 
Approximation by Vilenkin-Fourier sums in $L_{p}(G_{m})$, Acta Acad. Paed. Nyireg., 13 (1992), 35-39.

\bibitem{BNT9} \textit{I. Blahota and K. Nagy,} Approximation by $\theta$-means of Walsh-Fourier series, Anal. Math., 44, (2018), 57-71.

\bibitem{BNT10} \textit{I. Blahota, K. Nagy and G. Tephnadze,} Approximation by Marcinkiewicz $\theta$-means of double Walsh-Fourier series, Math. Inequal. Appl., 22, 3 (2019) 837-853.

\bibitem{BNPT} \textit{I. Blahota, K. Nagy, L. E. Persson and G. Tephnadze,} A sharp boundedness result concerning  maximal operators of Vilenkin-Fourier series on  martingale Hardy spaces, Georgian Math. J., 26, 3 (2019), 351-360. 

\bibitem{BPT1} \textit{I. Blahota, L. E. Persson and G. Tephnadze,} Two-sided estimates of the Lebesgue constants with respect to Vilenkin systems and applications, Glasg. Math. J., 60,  1 (2018), 17-34.

\bibitem{fi1} \textit{N. I. Fine,} On Walsh function, Trans. Amer. Math. Soc.
65 (1949), 372-414.

\bibitem{Fr1} \textit{S. Fridli,} Approximation by Vilenkin-Fourier series,
Acta Math. Hung., 47, 1-2 (1986), 33-44.

\bibitem{Gat1} \textit{G. G\'at,} Best approximation by Vilenkin-like systems,
Acta Acad. Paed. Nyireg., 17 (2001), 161-169.

\bibitem{ggt} \textit{G. G\'at, U. Goginava and G. Tkebuchava,} Convergence in
measure of logarithmic means of quadratical partial sums of double
Walsh-Fourier series, J. Math. Appl., 323,  1 (2006),  535-549.

\bibitem{gt} \textit{U. Goginava and G. Tkebuchava,} Convergence of subsequence of partial sums and logarithmic means of Walsh-Fourier series, Acta Sci. Math., 72 (2006), 159-177.

\bibitem{gol} \textit{B. I. Golubov, A. V. Efimov  and  V. A.
Skvortsov,} Walsh series and transforms. (Russian) Nauka, Moscow, 1987,
English transl, Mathematics and its Applications (Soviet Series), 64. Kluwer
Academic Publishers Group, Dordrecht, 1991.

\bibitem{Gul1} \textit{N. V. Guli\'cev,} Approximation to continuous functions by Walsh-Fourier series, Anal. Math., 6 (1980), 269-280.

\bibitem{luk} \textit{S. F. Lukomskii,} Lebesgue constants for characters of
the compact zero-dimensional Abelian groups, East J. Approx.,  15,  2
(2009), 219-231.

\bibitem{MST} \textit{N. Memi\'c, I. Simon and G. Tephnadze,} Strong convergence of two-dimensional Vilenkin-Fourier series, Math. Nachr., 289, 4 (2016) 485-500.

\bibitem{na1} \textit{K. Nagy,} Approximation by Ces\'aro means of negative
order of Walsh-Kaczmarz-Fourier series, East J. Approx., 16, 3 (2010), 
297-311.

\bibitem{NT1} \textit{K. Nagy and G. Tephnadze,} The Walsh-Kaczmarz-Marcinkiewicz means and Hardy spaces, Acta Math. Hung., 149, 2 (2016), 346-374.

\bibitem{Onn10} \textit{C.W. Onneweer,} On $L$-convergence of Walsh-Fourier
series, Internat. J. Math. Sci., 1 (1978), 47-56.

\bibitem{Pal11} \textit{A. Paley,} A remarkable series of orthogonal functions,
Proc. London Math. Soc. 34 (1932), 241-279.

\bibitem{PTW2} \textit{L. E. Persson, G. Tephnadze and P. Wall,} On an approximation of 2-dimensional Walsh-Fourier series in the martingale Hardy spaces, Ann. Funct. Anal., 9, 1 (2018), 137-150.

\bibitem{tepbook} \textit{L. E. Persson, G. Tephnadze and F. Weisz,} Martingale Hardy Spaces and Summability of  One-dimenional Vilenkin-Fourier Series, Birkh\"auser/Springer, 2022.

\bibitem{PSTW} \textit{L.-E. Persson, F. Schipp, G. Tephnadze and F. Weisz,} An analogy of the Carleson-Hunt theorem with respect to Vilenkin systems, J. Fourier Anal. Appl., 28, 48 (2022),  1-29.

\bibitem{sws} \textit{F. Schipp, W. Wade, P. Simon and J. P\'al,} Walsh series,
An Introduction to Dyadic Harmonic Analysis, Akademiai Kiado, Budapest-Adam-Hilger, Bristol-New-York, 1990.

\bibitem{Simon1} \textit{P. Simon,} A note on the of the Sunouchi operator with
respect to Vilenkin systems, Ann. Univ. Sci. Budapest. Eotvos. Sect. Math., 43 (2001), 101-116.

\bibitem{tep0} \textit{G. Tephnadze,}  On the Vilenkin-Fourier coefficients, Georgian Math. J., 20,  1 (2013), 169-177.

\bibitem{tep1} \textit{G. Tephnadze, } On the partial sums of Vilenkin-Fourier series, J. Contemp Math. Anal.,  49,    1 (2014), 23-32.

\bibitem{tep2} \textit{G. Tephnadze, }  On the partial sums of Walsh-Fourier series, Colloq. Math., 141, 2 (2015), 227-242.

\bibitem{tep3} \textit{G. Tephnadze, } On the convergence of partial sums with respect to Vilenkin system on the martingale Hardy spaces, J. Contemp. Math. Anal., 53,  5 (2018) 294-306.

\bibitem{tep5} \textit{G. Tephnadze, }  Convergence and Strong Summability of the two-dimensional Vilenkin-Fourier Series, Nonlinear Studies, 26, 4, (2019) 973-989.

\bibitem{tep8} \textit{G. Tephnadze, } Strong convergence of two-dimensional Walsh-Fourier series, Ukr. Math. J., 65, 6 (2013), 914-927.


\bibitem{We1} \textit{F. Weisz,} Martingale Hardy spaces and their
applications in Fourier Analysis, Springer, Berlin-Heideiberg-New York, 1994.

\bibitem{We3} \textit{F. Weisz,} Hardy spaces and Ces\'aro means of
two-dimensional Fourier series, Bolyai Soc. Math. Studies, 5 (1996), 353-367.
\end{thebibliography}
\end{document}